\documentclass[12pt,a4paper,dvipsnames]{article}

\usepackage[utf8]{inputenc}
\usepackage[english]{babel}
\usepackage{amsmath}
\usepackage{amsthm}
\usepackage{amssymb}
\usepackage{geometry}
\usepackage{a4wide}
\usepackage{bm}
\usepackage{microtype}
\usepackage{mathtools}
\usepackage{csquotes}

\usepackage{enumitem}
\setlist[enumerate]{itemsep=0mm,parsep=2mm}

\usepackage{xcolor}
\newcommand\myshade{85}
\colorlet{myurlcolor}{Aquamarine}
\usepackage{hyperref}
\hypersetup{
  linkcolor  = black,
  citecolor  = black,
  urlcolor   = myurlcolor!\myshade!black,
  colorlinks = true,
}
\usepackage[capitalise, noabbrev, nameinlink]{cleveref}
\crefname{equation}{}{}
\usepackage{tikz, graphicx, standalone, caption, subcaption, wrapfig}
\usetikzlibrary{calc,decorations.text}

\definecolor{edgeblack}{rgb}{0.25,0.25,0.25}
\definecolor{vertexblack}{rgb}{0.2,0.2,0.2}

\usepackage{chngcntr}
\usepackage{apptools}
\AtAppendix{\counterwithin{theorem}{subsection}}

\theoremstyle{definition}
\newtheorem{theorem}{Theorem}[section]

\newtheorem{lemma}[theorem]{Lemma}
\newtheorem*{lemma*}{Lemma}
\newtheorem*{conjecture*}{Conjecture}
\newtheorem*{lemma''*}{``Lemma''}
\newtheorem{claim}[theorem]{Claim}
\newtheorem*{claim*}{Claim}

\newtheorem{conjecture}[theorem]{Conjecture}

\begin{document}

\title{\bf Removable trees and matchings in $k$-connected and $k$-edge-connected graphs}

\author{Adam D.W. Clay\thanks{Department of Mathematics, Purdue University, West Lafayette, IN 47907, USA.
e-mail: {\tt adwclay@gmail.com}}\and
Tibor Jord\'an\thanks{Department of Operations Research, ELTE E\"otv\"os Lor\'and University, and the HUN-REN-ELTE Egerv\'ary Research Group
on Combinatorial Optimization, P\'azm\'any P\'eter s\'et\'any 1/C, 1117 Budapest, Hungary.
e-mail: {\tt tibor.jordan@ttk.elte.hu}.}
}

\date{August 4, 2026}

\maketitle

\begin{abstract}
T. Hasunuma (J. Graph Theory, 2023) conjectured that if
$G$ is a $k$-connected (resp. $k$-edge-connected) graph with minimum degree
$\delta(G) \ge k + m - 1$, and $T$ is a tree of order $m$, then
$G$ contains a removable copy of $T$, that is, a
subtree $T'$ isomorphic to $T$ such that $G - E(T')$ is $k$-connected (resp. $k$-edge-connected). We prove (a strengthening of) this conjecture.
We also consider removable matchings in graphs with high minimum degree. We show, among others, that if
$G$ is a $k$-edge-connected graph on at least $2m$ vertices with minimum degree
$\delta(G) \ge k + m$, then there exists a matching $M$ of size $m$ in $G$ for which $G-M$
is $k$-edge-connected. 
\end{abstract}

\section{Introduction}


We consider undirected simple graphs. 
The minimum degree of a graph $G$ is denoted by $\delta(G)$.
Our goal is to show that if $G$ is $k$-(edge-)connected and $\delta(G)$ is large enough,
then $G$ contains a subgraph $H$, of a given type, for which $G-E(H)$ is 
$k$-(edge-)connected. We shall frequently call such a subgraph {\it removable}.
In this paper we focus on the cases when $H$ is a tree, a forest, or a matching of given size.
For a survey on removable subgraph problems see \cite{TM}.

One of our main results is the following theorem concerning removable trees
in $k$-connected or $k$-edge-connected graphs.
The statement of the theorem was conjectured by T. Hasunuma \cite[Conjecture 1.5]{Hasunuma2023}, who also verified the special cases
$k=1,2$. The case $k=3$ has been settled by 
Liu, Liu, and Hong \cite{LLH}, and independently
Yang and Tian \cite{YT}.

\begin{theorem}
\label{thm:Hasunuma conjecture}
	Let $G$ be a $k$-(edge-)connected graph, and let $T$ be a tree of order $m$. If $\delta(G) \ge k + m - 1$, then $G$ contains a subtree $T' \cong T$ such that $G - E(T')$ is $k$-(edge-)connected.
\end{theorem}

We shall also prove an extension to removable forests with at most four components.

\begin{theorem} \label{thm:removable forests}
	Let $G$ be a $k$-(edge-)connected graph, and let $F$ be a forest of order $m$ with $\omega$ connected components, where $\omega \le 4$. If $\delta(G) \ge k + m - \omega$ and $k \ge \omega - 1$, then $G$ contains a subforest $F' \cong F$ such that $G - E(F')$ is $k$-(edge-)connected.
\end{theorem}

The above results will follow from the next strengthening of Theorem \ref{thm:Hasunuma conjecture}. In a graph $G=(V,E)$ the subgraph induced by a vertex set $X\subseteq V$
is denoted by $G[X]$. A {\it clique} of $G$ means a complete subgraph of $G$.

\begin{theorem} \label{thm:removable subtree missing X}
	Let $G = (V,E)$ be a $k$-(edge-)connected graph, let $T$ be a tree of order $m$, let $X \subsetneq V$, and let $q\geq 0$ be an integer. Suppose that 
    $k\geq q$, $G[X]$ contains a clique of size $q$, and 
    \begin{equation}
    \label{degreebound}
    d(v)\geq k + m - 1 + |X| - q
    \end{equation}
    for all $v\in V-X$.
       Then there exists a subtree $T' \cong T$ of $G - X$ such that $G - E(T')$ is $k$-(edge-)connected.
\end{theorem}

By choosing $X=\emptyset$ and $q=0$ we obtain
Theorem \ref{thm:Hasunuma conjecture} as an immediate corollary of Theorem \ref{thm:removable subtree missing X}. 
We can also deduce that the removable tree in Theorem \ref{thm:Hasunuma conjecture}
can be chosen so that it avoids a given vertex $v$ (put $X=\{v\}$
and $q=1$).

We also consider removable matchings, whose study was initiated in a recent paper of Li et al. \cite{li2026halinsedgeremovabilitymatching}. We have the following sufficient condition for the
existence of a removable matching of size two.
The $k$-connected version
of the statement was proved
earlier in
\cite[Theorem 1.6]{li2026halinsedgeremovabilitymatching}.

\begin{theorem} \label{thm:removable 2-matching}
	Let $G=(V,E)$ be a
    $k$-(edge-)connected graph, $k \ge 2$, with $\delta(G) \ge k + 1$.
    Then $G$
    contains a matching $M$ with $|M|=2$ such that $G - M$ is $k$-(edge-)connected.
\end{theorem}

The following theorem is our main result concerning $k$-edge-connected graphs of high minimum degree.
The $k$-connected version of Theorem \ref{thm:matching size}(a) appeared in \cite[Theorem 1.7]{CKP}.

\begin{theorem} \label{thm:matching size}
Let $G=(V,E)$ be a $k$-edge-connected graph and let $m$ be a positive integer. Suppose that $|V|\geq 2m$ and 
$\delta(G) \ge k + m$. Then\\
(a) $G$ contains a matching $M$ with $|M|=m$ such that $G-M$ is $k$-edge-connected.\\
Furthermore, if $k\geq 2$ and $1\leq m\leq k-1$, then\\
(b) $G$ contains a matching $M$ with $|M|=m+1$ such that $G-M$ is $k$-edge-connected.
\end{theorem}

The main tool in our proofs is the so-called maximum adjacency ordering of graphs.
In Section \ref{sec:vertex orderings of graphs} we summarize the key properties of
these orderings. We prove Theorems \ref{thm:Hasunuma conjecture}, \ref{thm:removable forests}, and \ref{thm:removable subtree missing X} in Section \ref{sec:removable trees}.
In Section \ref{sec:removable matchings}
we prove Theorems \ref{thm:removable 2-matching} and \ref{thm:matching size}.
Section \ref{sec:removable 3-matching} contains two further results on removable
matchings of size three.
Concluding remarks 
are given in
Section \ref{sec:conc}.

\section{Maximum adjacency orderings of graphs}
\label{sec:vertex orderings of graphs}

Let $G = (V,E)$ be a graph on $n$ vertices and let $v_1,v_2, \ldots, v_n$ be an ordering of its vertices. For an integer $i$, $1\leq i\leq n$, let $V_i = \{v_1,v_2, \ldots, v_i\}$.
We say that $v_1,v_2, \ldots, v_n$ is a \textit{maximum adjacency ordering}
(or {\it MA ordering}, for short) of $G$ if
\begin{equation}
\label{eq:max-back criteron}
	d_G(v_i,V_{i-1}) \ge d_G(v_j, V_{i - 1}) 
\end{equation}
for all pairs $i,j$, such that $2 \le i < j \le n$, where $d_G(v,X)$ denotes the number of edges from vertex $v$ to vertex set $X\subseteq V$ in $G$.

Let $G=(V,E)$ be a graph and let $v_1,v_2, \ldots, v_n$ be an MA ordering of $G$. For a subgraph $G'$ of $G$ 
let $d_{G'}^- (v_i) = d_{G'}(v_i, V_{i - 1})$, $1\leq i\leq n$, denote the number of
left neighbours of $v_i$ in $G'$, with respect to the ordering. For a positive integer $w$, let
$E_w^{i}$ be the set of edges that connect $v_i$ to its 
$\min \{w, d_G^-(v_i)\}$ leftmost neighbours with respect to the ordering
in $G$.
Let $E_w=\cup_{i=1}^n E_w^i$ and $G_w=(V,E_w)$.

\begin{lemma}
\cite{FIN,NIbook}
\label{lem:maxback}
Let $G=(V,E)$ be a $k$-(edge-)connected graph
and let $v_1,v_2, \ldots, v_n$ be an MA ordering of $G$.
Then $G_k=(V,E_k)$ is a $k$-(edge-)connected
spanning subgraph of $G$. Furthermore, $v_1,v_2, \ldots, v_n$ is an MA ordering of $G-E_k$.
\end{lemma}

We shall also use a refinement of Lemma \ref{lem:maxback} in the
proof of the $w=4$ case of Theorem \ref{thm:removable forests} and in
Sections \ref{sec:removable matchings} and \ref{sec:removable 3-matching}. 
Let $G=(V,E)$ be a graph and let $v_1,v_2, \ldots, v_n$ be an ordering of $V$.
The {\it forest of the ordering} is a subgraph $F$ of $G$ which consists of those edges that
connect a vertex
$v_i$, $1\leq i\leq n$, to its leftmost neighbour (if $v_i$ has at least one left neighbour).
Note that $F$ is indeed cycle-free.
Suppose that the ordering is an MA ordering of $G$.
Then the forest $F$ of the ordering is a maximal forest of $G$, and 
the ordering is an MA ordering of $G-E(F)$, too.
For some positive integer $w$ 
let $F_1$ be the forest of the ordering and let $F_i$ be the forest of the ordering in $G-\cup_{j=1}^{i-1} E(F_j)$,
$2\leq i\leq w$. Note that $E_w=\cup_{j=1}^{w} E(F_j)$.

The following strengthening of Lemma \ref{lem:maxback} will make it possible to modify the edge set of $F_k$ (and $E_k$) so that $k$-(edge-)connectivity
is preserved. This tool will be useful in the identification of removable matchings.
The $k$-edge-connected version of the next lemma
follows easily from the fact that the forests $F_i$, $1\leq i\leq k-1$ are also maximal. 
The $k$-connected version 
can be proved by a straightforward modification of the inductive proof of
\cite[Theorem 7.5.7]{FAbook}. 


\begin{lemma}
\label{lem:swap}
Let $G=(V,E)$ be a $k$-(edge-)connected graph
and let $v_1,v_2, \ldots, v_n$ be an MA ordering of $G$.
Suppose that $F_k'$ is a maximal forest in $G-E_{k-1}$.
Then $G_k'=(V,E_{k-1}\cup F_k')$ is a $k$-(edge-)connected
spanning subgraph of $G$.
\end{lemma}

We refer the reader to \cite{FAbook,NIbookbook} for more details and applications of
MA orderings, and to \cite{BJ} for an earlier application of these orderings in removable subgraph problems.

\section{Removable trees}
\label{sec:removable trees}

In this section we
prove Theorem \ref{thm:removable subtree missing X}
and show how it implies Theorem \ref{thm:removable forests}.
%
On the reader's first read of the proof (or to obtain Theorem \ref{thm:Hasunuma conjecture} directly), it may be helpful to consider the case when $X = \emptyset$ and $q=0$. 

\begin{proof}[Proof of Theorem \ref{thm:removable subtree missing X}]
Let $v_1,v_2, \ldots, v_q \in X$ be the vertices of a clique of size $q$ in $G[X]$, ordered arbitrarily, and let $v_1,v_2, \ldots, v_n$ be an MA ordering of $G$ that extends this ordering. Let $G'=G-E_k$ and $H = G' - X$.
We shall develop an algorithm that outputs
a sequence $W_1 \subseteq \cdots \subseteq W_m = V(T)$ with $|W_i| = i$,
$1\leq i\leq m$, and
an injective map $\phi: V(T) \to V - X$ such that 
$H[\phi(W_i)]$ contains a copy of $T[W_i]$ for all $1\leq i\leq m$.
By Lemma \ref{lem:maxback}, if $T'$ is a copy of $T$ contained in $H[\phi(V(T))]$ then  
$G-E(T')$ is $k$-(edge-)connected, that is,
$T'$ is the desired removable tree in $G-X$.
For convenience, we will write $W_i' = \phi(W_i)$ and $t' = \phi(t)$ for $1\leq i\leq m$ and $t \in W_i$.

The algorithm has $m$ iterations. In each iteration we define $W_i$
by adding a new vertex to $W_{i-1}$, and also define the image of the new vertex (we put $W_{0}=\emptyset$). In the first iteration
we first choose a root vertex $r \in V(T)$ arbitrarily. Let $r'$ be the rightmost vertex of $H$ with respect to the MA ordering of $G$. 
Note that $r'$ exists, since $X \ne V$.
We define $W_1 = \{r\}$ and $\phi(r) = r'$.
During iteration $i \ge 2$ of the algorithm, we say a vertex $v$ of $G$ is {\it free} if $v \in V - X - W_{i - 1}'$.
In iteration $i$, $i \ge 2$, we consider a vertex 
$t \in W_{i- 1}$ which has a neighbor $s \notin W_{i - 1}$ in $T$.
We pick the rightmost free left neighbour $s'$ of $t'$ in $G$, and we define $W_i = W_{i - 1} \cup \{s\}$ and $\phi(s) = s'$.

Since each new assignment is made ``along'' a copy of an edge of $T$ in $H$, we have that $H[W_i']$ contains a copy of $T[W_i]$ for $1\leq i\leq m$.
It remains to show the above algorithm is well-defined. That is, for each iteration $i \ge 2$ and choice of $t$, a free left neighbour of $t'$ exists.

Let $N^-_v$ denote the vertices of $G$ which are to the left of $v$ in the MA ordering, for all $v\in V$. 
It follows from the definition of $E_k$ that $v_1, \ldots, v_k$ are isolated vertices of $G'$.
Since $q \le k$ and $V_{q}\subseteq X$, the correctness of the algorithm is guaranteed by the following claim.

\begin{claim} 
After iteration $i$ of the algorithm we have
	\begin{equation}
		\label{eq:G' back degree}
		d^-_{G'} (t') \ge m - i + |N^-_{t'} \cap W_i'| + |N^-_{t'} \cap X| - q,
	\end{equation}
	for all $1\leq i< m$ and $t \in W_i$.
\end{claim}

\begin{proof}
	The proof is by induction on $i$.
	After the first iteration we have $W_1=\{r\}$ and $\phi(r)=r'$.
	Suppose there are $\alpha$ vertices to the right of $r'$ in $G$. Each of these vertices belongs to $X$ by the choice of $r'$. 
	Since $r'$ has at most $k$ left neighbours in $G_k$ by the definition of $E_k$, we have $d^-_{G'} (r') \ge d_G (r') - \alpha - k$.
	Moreover, we have
	$|N^-_{r'} \cap W_1'| = 0$ and
	$|N^-_{r'} \cap X| = |X| - \alpha$. Thus \eqref{degreebound} gives
	$$d^-_{G'} (r') \ge k+m-1+|X|-q - \alpha - k = m-1+|N^-_{r'} \cap W_1'|+|N^-_{r'} \cap X|-q.$$
	Therefore  \eqref{eq:G' back degree} holds for $i = 1$.
	
	Next suppose $2 \le i < m$ and \eqref{eq:G' back degree} holds up to $i-1$.
	First observe that every $u \in W_{i - 1}$ satisfies \eqref{eq:G' back degree}, since $u$ satisfies \eqref{eq:G' back degree} with respect to $i-1$ by induction, and
    $|N^-_{u'} \cap W_{i}'|\leq |N^-_{u'} \cap W_{i-1}'| + 1$.
    Hence it suffices to show that \eqref{eq:G' back degree} is satisfied for the vertex $s \in W_i - W_{i - 1}$.
	Let $t \in W_{i - 1}$ be the neighbour of $s$ in $T[W_{i-1}]$.
	Then 
    $s'$ is the rightmost free left neighbor of $t'$ at the start of iteration $i$.
	Let $Y$ denote the left neighbours of $t'$ in $G'$
    which do not precede $s'$ in the ordering.
	Then every vertex in $Y$ must be in either $W_i'$ or $X$.
	Hence 
	\begin{equation}
		\label{w}
		|Y|= (|N^-_{t'} \cap W_i'| - |N^-_{s'} \cap W_i'|) + (|N^-_{t'} \cap X| - |N^-_{s'} \cap X|).
	\end{equation}
	Lemma \ref{lem:maxback} implies that the MA ordering of $G$
	is also an MA ordering of $G'$.
	Thus we can apply (\ref{eq:max-back criteron}) to $s'$ and $t'$ in $G'$
	to obtain
	\begin{equation}
		\label{MAappl}
		d^-_{G'}(s') \geq d^-_{G'}(t') - |Y| .
	\end{equation}
	By using (\ref{w}), (\ref{MAappl}), and that $t'$ satisfies \eqref{eq:G' back degree}, a simple calculation gives that $s'$ also satisfies \eqref{eq:G' back degree}, as claimed.
\end{proof}

It follows that the algorithm is indeed able to
assign a vertex of $V - X$ to each vertex of $T$ and find
a removable copy of $T$ in $H$.
This completes the proof of the theorem.
\end{proof}

We next show how Theorem \ref{thm:removable subtree missing X} implies Theorem \ref{thm:removable forests} on removable forests $F$ of order $m$. In what follows we shall assume, without loss of generality, that
each component of $F$ has at least two vertices.


\begin{proof}[Proof of Theorem \ref{thm:removable forests}]
	As we noted above, the $\omega = 1$ case is equivalent to Theorem \ref{thm:Hasunuma conjecture} and follows from Theorem \ref{thm:removable subtree missing X} by choosing $X = \emptyset$ and $q=0$. Let us consider the case $\omega = 2$. 
    Let $T_1, T_2$ be the connected components of $F$ with order
    $m_i$, $i=1,2$. Thus $m=m_1+m_2$. Let us
    apply Theorem \ref{thm:removable subtree missing X} to $G$ with
    $T_1$, $X = \emptyset$, $q=0$.
    Since $m_1\leq m-1$, the degree condition is satisfied, and we obtain
    a removable subtree $T_1' \cong T_1$ in $G$. Next we apply 
    Theorem \ref{thm:removable subtree missing X} 
    to $G - E(T_1')$ with $T_2$, $X = V(T_1')$ and $q=1$.
    Again, the degree condition holds for all $v\in V-X$, and we obtain 
    a removable subtree $T_2' \cong T_2$ in $G-X$.
    It follows that $F' = T_1' \cup T_2'$ is the required removable forest in $G$.
	
	For the $\omega = 3$ case, let $T_1, T_2, T_3$ be the connected components of $F$ with order
    $m_i$, $1\leq i\leq 3$. Thus $m=m_1+m_2+m_3$. 
    Let $T_{12}$ be the tree obtained from $T_1$ and $T_2$ by adding an edge $e$ connecting them. First we apply Theorem \ref{thm:removable subtree missing X} to $G$ with $T_{12}$,
    $X = \emptyset$, $q=0$. Since $m_3\geq 2$, the degree condition is satisfied and we obtain  
    a removable subtree $T_{12}' \cong T_{12}$ in $G$.
    Let $T_1', T_2' \subseteq T_{12}'$ be the subtrees corresponding to $T_1$ and $T_2$, and let $e'$ be the edge corresponding to $e$.
    Next we apply Theorem \ref{thm:removable subtree missing X} to 
    $G'=G - E(T_1')-E(T_2')$ with $T_3$, $X = V(T_1')\cup V(T_2')$ and $q\geq 2$.
 Note that $X$ contains a clique of size $2$ in $G'$ as $G'[X]$ contains $e'$.
    Hence there exists a removable subtree $T_3' \cong T_3$ in $G'-X$. Then $F' = T_1' \cup T_2' \cup T_3'$ is the required removable forest in $G$.

    Finally, we consider the case $\omega=4$.
    Let $T$ be a designated component of $F$ and let
    $F_i$, $1\leq i\leq 3$ be the other components.
    By applying the $\omega=2$ case of the theorem to the forest $F^+$ obtained
    from $F-V(T)$ by adding an edge $e$ that connects
    $F_1$ and $F_2$, we obtain a copy $F''$ of $F-V(T)$ in $G$ for which $G'=G-E(F'')$ is $k$-(edge-)connected and
    $X=V(F'')$ induces the edge $e'=x_1x_2$ corresponding to edge
    $e$ in $G'$.  
   Let 
   $m' = |V(T)|$ and $m'' = |V(F'')|$. 
   Thus $m=m'+m''$. Since $m'\geq 2$, there exists a leaf $t$
   and a vertex $r$ adjacent to $t$ in $T$.

Let $x_1,x_2,v_3,\dots,v_n$ be an MA ordering of $G'$.
    By running 
	the algorithm described in the proof of Theorem \ref{thm:removable subtree missing X} 
    with target tree 
    $T-t$ (of order $m'-1$), forbidden set $X$, and $q=2$,
    choosing this MA ordering of $G'$ and root vertex $r$ in $T-t$, 
    we obtain a subtree
    $T' \cong T - t$ of $G'-X$ for which $G'-E(T')$ is
    $k$-(edge-)connected.
      If $r$ has an additional free left neighbor $u$ after identifying $T'$, we can set $T'' = T' + ru$ and $F' = F'' + T''$, and we are done.
      Suppose $r$ has no free left neighbours.
      Let $X' = X - \{x_1, x_2\}$.
	Then $r$ must have degree $d_G (r) = k + m - 4 = k + (m'' - 2) + (m' - 2)$, and must be adjacent to all $m'' - 2$ vertices of $X'$, all $m' - 2$ vertices of $V(T') - r$, and an additional $k$ left neighbors in $V - V(T') - X'$ in $G$. Since $k \ge 3$, the left neighbor $u$ of $r$ in $F_k$ is different from $x_1$ and $x_2$. We again set $T'' = T + ru$ and $F' = F'' + T''$. We claim that $T''$ is a removable subtree of $G'$, which will complete the proof of the $\omega=4$ case.
	
	Since $|X| = m'' \ge 6$ and $|X'| = |X| - 2$, we have $X' \ne \emptyset$. For each $x \in X'$ we have $d_{G'} (x, V - X') \ge k + m - 4 - (|X'| - 1) = k + m' - 1 \ge k + 1$. Therefore $r$ is a leaf in $F_k$ since every vertex to its right belongs to $X'$ and hence has at least $k$ neighbors to the left of $r$. Then $F_k - ru + rx$ is a maximal forest in $G'-E(F_k)$
    for any $x \in X'$, so $G' - E(T'')$ is $k$-(edge-)connected by Lemma \ref{lem:swap}, as claimed.
\end{proof}

We close this section with another generalization of Theorem \ref{thm:Hasunuma conjecture}.
Let $\kappa_G(u,v)$ (resp. $\lambda_G(u,v)$) denote the maximum number of pairwise internally dijoint (resp. edge-disjoint) $u$-$v$ paths in $G$.
A spanning subgraph $H=(V,F)$ of a graph $G=(V,E)$ is called
a {\it local $k$-connectivity certificate} if
$$\kappa_H(u,v)\geq \min \{\kappa_G(u,v), k\}$$
for all $u,v\in V$. The definition of a {\it local $k$-edge-connectivity certificate} is similar,
by using $\lambda$ in place of $\kappa$. 
Note that if $G$ is $k$-(edge-)connected, then a local $k$-(edge-)connectivity certificate is a $k$-(edge-)connected spanning subgraph.
It is known that Lemma \ref{lem:maxback} holds in a more general form: $G_k=(V,E_k)$ is a local $k$-(edge-)connectivity certificate for all graphs $G$, 
see \cite{FIN,NIbook}.
Thus our results hold in this more general form, too.
In particular, we obtain the following strengthening of Theorem \ref{thm:Hasunuma conjecture}.

\begin{theorem}
	Let $G$ be a graph, and let $T$ be a tree of order $m$. If $\delta(G) \ge k + m - 1$, then $G$ contains a subtree $T' \cong T$ such that $G - E(T')$ is a local $k$-(edge-)connectivity certificate.
\end{theorem}

\section{Removable matchings}
\label{sec:removable matchings}

In this section we prove Theorems \ref{thm:removable 2-matching} and \ref{thm:matching size}.
The {\it degree} $d_G(X)$ of a vertex set $X$ in a graph $G$ is the number of edges from $X$ to $V-X$ in $G$.
In the proof of the next lemma we shall use the well-known fact that the degree function of a graph satisfies the submodular inequality. 
Namely, for all pairs $X,Y\subseteq V$ of vertex sets in a graph
$G=(V,E)$ we have
\begin{equation}
%
\label{submod}
d_G(X)+d_G(Y) \geq  d_G(X\cap Y)+d_G(X\cup Y).
\end{equation}

\begin{lemma}
\label{lem:sub}
Let $G=(V,E)$ be a $k$-edge-connected graph and $U\subseteq V$ with $E(G[U])\not= \emptyset$.
Suppose that $G-e$ is not $k$-edge-connected for all $e\in E(G[U])$ and let $A\subsetneq V$ be a minimal
vertex set subject to $d_G(A)=k$ and that there exists an edge $e\in E(G[U])$ from $A$ to $V-A$.
Then for each edge $f=uv\in E(G[U])$ with $u\in A$ and $v\in V-A$, $u$ is an isolated vertex in $G[A\cap U]$.
\end{lemma}

\begin{proof}
For a contradiction suppose that $uw\in E(G[A\cap U])$. Since $G-uw$ is not $k$-edge-connected, there exists a set $X\subsetneq V$ with $d_G(X)=k$ and such that $uw$ leaves $X$. 
We may assume $w\in X$.
By the minimality of $A$ we have $V-(A\cup X)\not= \emptyset$.
Then (\ref{submod}) and $k$-edge-connectivity imply
$$k+k = d_G(A) + d_G(X) \geq d_G(A\cap X) + d_G(A\cup X) \geq k+k,$$
which implies that $d_G(A\cap X)=k$.
Since $u\in A-X$, it contradicts the minimality of $A$.
%
\end{proof}

Lemma \ref{lem:sub}, combined with a minimum degree lower bound, gives the following corollary.
Note that by putting $U=V$ 
we obtain one of the basic results in this area: every minimally $k$-edge-connected graph has a vertex of degree exactly $k$ \cite{Lick}.

\begin{lemma} \label{lem:removable edge avoiding vertices}
	Let $G = (V, E)$ be a $k$-edge-connected graph and let $U \subseteq V$. 
    Suppose $E(G[U]) \ne \emptyset$ and 
   \begin{equation}
    \label{eq:min degree of U}
		d(u) \ge k + \lceil{\frac{|V-U| + 1}2}\rceil
	\end{equation}
    for every $u \in U$. Then $G[U]$ contains an edge $e$ such that $G - e$ is $k$-edge-connected.
\end{lemma}

\begin{proof}
Suppose no edge of $G[U]$ is removable in $G$. Then there exists a minimal
set $A\subsetneq V$ such that $d_G(A)=k$ and that there exists an edge $e\in E(G[U])$ from $A$ to $V-A$.
We may assume that $|A-U|\leq \lfloor \frac{|V-U|}{2}\rfloor$.
Let $f=uv\in E(G[U])$ with $u\in A$ and $v\in V-A$.
By Lemma \ref{lem:sub} $u$ is an isolated vertex in $G[A\cap U]$.
Hence $d_G(u) \le k + |A-U| \le k + \lfloor \frac{|V-U|}{2}\rfloor$, 
which contradicts \eqref{eq:min degree of U}.
\end{proof}

We are ready to prove the first main result of this section on removable matchings, which
extends the above mentioned theorem of Lick \cite{Lick} on the existence of a removable edge
in a $k$-edge-connected graph with minimum degree $k+1$ to
matchings of arbitrary size.
Although our focus is on $k$-edge-connected graphs, 
the high level approach is similar to that of \cite[Theorem 1.7]{CKP}.

\begin{proof} [Proof of Theorem \ref{thm:matching size}(a).]
	Let $M$ be a (possibly empty) removable matching of $G$, and suppose $s = |M| \leq m-1$. We shall show there exists a larger removable matching. Let $H = G - M$, $W = V(M)$, and $U = V - W$. Then $H$ is $k$-edge-connected and $|W| = 2s$. If $G[U]$ contains an edge then, since 
		$$\delta(G)
		\ge k + m
		= k + \lceil{\frac{(2m - 2) + 1}2}\rceil
		\ge k + \lceil{\frac{2s + 1}2}\rceil ,$$
	we can obtain a larger matching by adding an edge of $G[U]$ to $M$ by Lemma \ref{lem:removable edge avoiding vertices}.
		
	Suppose $G[U]$ contains no edge. Since $|U| = |V| - 2s \ge 2m - 2 (m - 1) = 2$, there exist distinct vertices $u,v \in U$. Then $N(u), N(v) \subseteq W$. If $|N(u) \cap V(e)| + |N(v) \cap V(e)| \le 2$ for every $e \in M$, then \(
		2 (k + m) \le 2 \delta(G) \le d(u) + d(v) \le |V(M)| = 2s \le 2m - 2 \)
	which is a contradiction. Thus, there exists $xy \in M$ such that $ux, vy \in E$.
	
	We claim that $H - \{u, v\}$ is $k$-edge-connected. Suppose not. Then there exists $\emptyset \ne X \subsetneq V - \{u, v\}$ such that $d_{H - \{u, v\}} (X) \le k - 1$. By symmetry we may assume $|X \cap W| \le |W| / 2 \le m - 1$. If $X$ contains some $z \in U$ then, since $U$ is independent, we have $d_G(z) \le |X \cap W| + k - 1 \le k + m - 2$, a contradiction. Otherwise $X \subseteq W$, in which case for any $z \in X$ we have $d_G(z) \le 2 + (|X| - 1) + (k - 1) \le k + m - 1$, which also contradicts our minimum degree assumption. Thus $H - \{u, v\}$ is indeed $k$-edge-connected, as claimed.
	
    
    Since $d(u), d(v) \ge k + 1$ and the addition of a new vertex of degree at least $k$ to a $k$-edge-connected graph preserves $k$-edge-connectivity, it follows that  $H - ux - vy$ is $k$-edge-connected. Thus $M \cup \{ux, vy\} - xy$ is a larger removable matching in $G$.
\end{proof}

Theorem \ref{thm:matching size}(b) will be proved by induction, using Theorem \ref{thm:removable 2-matching} as the base case, so we continue with the proof of this result. The $k$-connected version
of the statement was proved
earlier in
\cite[Theorem 1.6]{li2026halinsedgeremovabilitymatching}.
We shall use the following results, due to Mader.

\begin{theorem} \cite{M,M72}
\label{lem:vertices of degree k in minimally connected graphs}
	Suppose $G$ is a minimally $k$-(edge)-connected graph. Then $G$ has at least 
    $k+1$
vertices of degree $k$.
\end{theorem}

\begin{proof}[Proof of Theorem \ref{thm:removable 2-matching}]
	Let $e = xy$ be an edge of $G$ for which $H = G - e$ is $k$-(edge-)connected. First suppose that
    $d_H (x) = d_H (y) = k$. Theorem \ref{lem:vertices of degree k in minimally connected graphs} and the assumption $k\geq 2$ imply that
    $G$ has at least 3 vertices of degree $k$. Therefore there exists an edge $f$ in $G-e$
    for which $G-\{e,f\}$ is $k$-(edge-)connected. As $f$ must be disjoint from $e$, the theorem follows.

    Hence we may assume that $d_H (y) \ge k + 1$.  Let $x, v_2, \ldots, v_n$ be a maximum adjacency ordering of $H$ starting with $x$, in which $y$ appears as early as possible. 
    If $v_n = y$, then we must have $d_H^- (v_{n - 1}) > d_H (y, V_{n - 2}) \ge k$. In this case, $v_{n - 1}$ is incident to a removable edge going to the left which is not adjacent to $e$. These edges together form the required removable matching of size two.

Finally, suppose that  $v_n \ne y$.
Then $d^-_H (v_n) \ge k + 1$.
Let $f=wv_n$ be an edge incident with $v_n$ in $G-E_k$. 
Then $G-f$ is $k$-(edge-)connected and 
$w\not= x$ 
by the definition of $E_k$,
Lemma \ref{lem:maxback}, and the choice of 
the ordering.
%
    If $w\not=y$, we are done. 
    Otherwise $f = yv_n$. Let $g$ be the edge incident to $v_n$ in $F_k$. Since $v_n$ is a leaf of $F_k$, 
    $F_k - g + f$ is a maximal forest in $G - E_{k - 1}$. Thus $H-g$ is $k$-(edge-)connected by Lemma \ref{lem:swap}, and $g$ cannot be adjacent to $e$ since $k \ge 2$ and $F_1$ contains all edges incident to $x$.
    Hence $M=\{e,g\}$ is a removable matching of size two in $G$.
\end{proof}

Note that 
the proof 
is substantially shorter than the proof given in 
\cite{li2026halinsedgeremovabilitymatching}
for $k$-connected graphs.

We next verify the second part of Theorem \ref{thm:matching size}, which improves on the first
part in the case when $k\geq 2$ and $\delta(G)\leq 2k-1$.
We recall a further result of Mader.

\begin{theorem} \cite{M86}
\label{thm:k-EC removable vertex}
	Every $k$-edge-connected graph $G$ with $\delta(G) \ge k + 1$ contains a vertex $v$ such that $G - v$ is $k$-edge-connected.
\end{theorem}

\begin{proof}[Proof of Theorem \ref{thm:matching size}(b).]
	We proceed by induction on $m$. The statement is true for the base case $m=1$ by Theorem \ref{thm:removable 2-matching}. Let $G$ be a $k$-edge-connected graph with 
    $\delta(G) \ge k+m$, for some $2\leq m\leq k-1$. By Theorem \ref{thm:k-EC removable vertex}, there exists a vertex $v$ of $G$ such that $H = G - v$ is $k$-edge-connected. Then 
    $\delta(H) \ge \delta(G) - 1 \ge k+m - 1$, and hence
there exists a removable matching $M$ in $H$ of size $m$ by induction. Since $m\le k - 1$, we have
	\begin{equation}
		\delta(G) \ge k+m \geq 2m + 1 ,
	\end{equation}
	so there exists a neighbour $u$ of $v$ in $G$ such that $u \notin V(M)$. Since adding a vertex of degree at least $k$ to a $k$-edge-connected graph preserves $k$-edge-connectivity, $G - M - uv$ is $k$-edge-connected. Thus $M \cup e$ is a removable matching of $G$ of size $m + 1$, as required.
	\end{proof}

\section{Removable matchings of size three}
\label{sec:removable 3-matching}

In this section, we explore removable 3-matchings in more detail.
It is known that if $G$ is a $k$-connected graph on at least six vertices with
$\delta(G)\geq k+2$, then $G$ has a removable matching of size three, 
see \cite[Theorems 1.7, 1.9]{li2026halinsedgeremovabilitymatching}
(for $k\leq 3$) and  
\cite[Corollary 5.1]{CKP} (for $k\geq 4$).
Theorem \ref{thm:matching size}(b) implies a similar result for $k$-edge-connected graphs,
assuming $k\geq 3$.
We extend this result to the $k=2$ case.
We also give a
new, shorter proof for the above mentioned results concerning $k$-connected graphs, for $k\geq 3$.
Both proofs rely on MA orderings.

\begin{theorem} \label{thm:f'(2,4)}
Let $G=(V,E)$ be a 2-edge-connected graph with $|V|\geq 6$ and $\delta(G)\geq 4$. Then $G$ contains a matching $M$ with
$|M|=3$
such that $G-M$ is 2-edge-connected.
\end{theorem}

\begin{proof}
	Let $G$ be 2-edge-connected with $\delta(G) \ge 4$. Since the statement holds for 2-connected graphs 
    by \cite{li2026halinsedgeremovabilitymatching}, we may assume $G$ is not 2-connected. Let $v$ be a cut vertex of $G$, and let $W \subseteq V - v$ be the vertex set of a component of $G - v$. Then $H = G[W \cup v]$ is 2-edge-connected. We claim we can find a removable matching of size two in $H$ that avoids $v$.

    Let $r=|V(H)|$. Note that every vertex of $H$, except possibly $v$, has degree at least 4. Thus
    $r\geq 5$. Consider an MA ordering $v, v_2, \ldots, v_r$ of $H$ starting with $v$. 
    Let $x$, $y$, and $z$ be the three rightmost neighbours of $v_n$ in order, and let $w$ be the rightmost left neighbour of $z$. We have $d^-_H (z) \ge 3$ by \eqref{eq:max-back criteron}, so $yv_n, zv_n, wz \notin E_2$. If $w \ne y$, then $\{wz, yv_n\}$ is a removable matching of size two avoiding $v$. Otherwise, we claim $\{yz, xv_n\}$ is our desired removable matching. This is clear if $xv_n \notin E_2$. If $xv_n \in E_2$, then $xv_n \in E(F_2)$ and $F_2 - xv_n + zv_n$ is a maximal forest of $G - E_1$ since $v_n$ is a leaf of $F_2$. Hence $\{yz, xv_n\}$ is removable by Lemma \ref{lem:swap}. This proves the claim.
	
	Since $v$ is a cut vertex, $G - v$ has at least two components. By the above claim we can find a removable matching of size two in each component that avoids $v$. Thus their union is a removable matching of size at least four in $G$. This completes the proof.
\end{proof}

To prove the last result of this section, 
we need one more result of Mader. We say that an edge $e$ of a $k$-connected graph $G$ is \textit{$k$-essential} if $G - e$ is not $k$-connected.

\begin{theorem} \cite{M72}
\label{thm:cycles in k-connected graphs}
	Let $G$ be a $k$-connected graph and $C$ be a cycle consisting of $k$-essential edges in $G$. Then there exists a vertex in $V(C)$ of degree $k$ in $G$.
\end{theorem}

\begin{theorem} \label{thm:removable 3-matching}
	Let $G = (V, E)$ be a $k$-connected graph, $k \ge 3$, with $\delta(G) \ge k + 2$. Then $G$ contains a matching $M$ with $|M|=3$ such that $G - M$ is $k$-connected.
\end{theorem}

\begin{proof}
	We first show we can find a removable matching $M' = \{x_1 x_2, y_1 y_2\}$ such that $x_1 y_1 \in E$. Let $v_1, \ldots, v_n$ be an MA ordering of $G$, and consider a longest $v_{n - 1}$-$v_n$ path $P$ in $G' = G - E_k$. Since $\delta(G) \ge k + 2$, $P$ has at least one internal vertex. If $P$ has two or more internal vertices, we may choose $M'$ to be the first and third edge of $P$. Otherwise, $d(v_n) = d(v_{n - 1}) = k + 2$ and $v_j, v_{n - 1}, v_n$ form a triangle in $G'$ where $v_j$ is the internal vertex of $P$. If $u$ is the neighbour of $v_n$ in $F_k$, then $F_k - uv_n + v_{n - 1} v_n$ is a maximal forest of $G - E_{k - 1}$ by Lemma \ref{lem:swap}. Hence $M' = \{uv_n, v_j v_{n - 1}\}$ is our desired removable matching of size two.
	
	Let $x_1, y_1, u_3, \ldots, u_n$ be an MA ordering of $G - M'$, and let $X = V(M')$. If the rightmost vertex $v \in V - X$ has at least $k + 1$ neighbors in $V - \{x_2, y_2\}$, then we may choose the rightmost neighbor $z$ of $v$ in $V - X$ and set $M = M' \cup zv$. Otherwise, $d(v) = k + 2$ and $v$ is adjacent to both $x_2$ and $y_2$. Let $u$ be the neighbour of $v$ in $F_k$. If $v$ is not the parent of one of $x_2$ or $y_2$ in $F_k$, say $y_2$, then $F_k - uv + vy_2$ is a maximal forest of $G - E_{k - 1}$, so we may set $M = M' \cup uv$
    by Lemma \ref{lem:swap}. Thus, we may assume $d(v) = d(x_2) = d(y_2) = k + 2$, $v$ precedes $x_2$ and $y_2$ in the ordering, and $v, x_2, y_2$ forms a triangle.

Let $H = G - M'$ and $Y = V - X$. Let $m_y = |E(H[Y])|$, $m_x = |E(H[X])|$, and $m_{xy} = |E(X, Y)|$. Suppose there exists no edge $e$ of $H[Y]$ such that $H - e$ is $k$-connected. By symmetry of the earlier argument, we may assume $x_1 y_1, x_2 y_2 \in E(H)$ and $d_H (x_1) = d_H (x_2) = d_H (y_1) = d_H (y_2) = k + 1$. Then $m_x \ge 2$ so
    \begin{equation} \label{eq:mxy}
		m_{xy}
		= \sum_{v \in X} d(v) - 2 m_x
		= 4k + 4 - 2 m_x
		\le 4k
    \end{equation}
	and
    \begin{equation} \label{eq:2my+mxy}
		2 m_y + m_{xy}
		= \sum_{v \in Y} d(v)
		\ge (k + 2) (n - 4) .
    \end{equation}
	By Theorem \ref{thm:cycles in k-connected graphs}, $H[Y]$ must be a forest, so $m_y \le n - 5$. Combining this with \eqref{eq:mxy} and \eqref{eq:2my+mxy} yields $n \le 8 - \frac2k$, and hence $n \le 7$. 
    Hence either $k=4$ and $G=K_7$, or $k=3$ and $G$ can be obtained from $K_7$ be deleting a matching. 
    In the former case every matching of size three is removable. In the latter case
    a simple case analysis shows that $G$ has a matching
    $M$ of size three for which 
$G - M$ contains a 3-connected spanning subgraph isomorphic to $K_{3,4}$.
\end{proof}

\section{Concluding remarks}
\label{sec:conc}

Concerning the algorithmic aspects of removable forests, we recall that
an MA ordering of a graph $G=(V,E)$ and the spanning subgraph $G_k$, for a given $k\geq 1$,
can be found in $O(|V|+|E|)$ time \cite{NIbook}. Thus the algorithm described in the proof of Theorem \ref{thm:Hasunuma conjecture} gives rise to 
polynomial time algorithms for finding the removable trees and forests in Section \ref{sec:removable trees}.

It would be interesting to see whether Theorem \ref{thm:removable forests} extends
to all forests. 
We conjecture that it does.


\begin{conjecture} \label{conj:removable forests}
	Let $G$ be a $k$-(edge-)connected graph, and let $F$ be a forest of order $m$ with $\omega$ connected components. If $\delta(G) \ge k + m - \omega$ and $k \ge \omega - 1$, then $G$ contains a subforest $F' \cong F$ such that $G - E(F')$ is $k$-(edge-)connected.
\end{conjecture}

\section{Acknowledgements}

This work was supported by the National Research, Development and Innovation Office
of Hungary, grant no. Advanced 152786, and the
MTA-ELTE Momentum Matroid Optimization Research Group.

\end{document}